\documentclass[11pt,a4paper]{amsart}

\usepackage[a4paper,margin=1in]{geometry}

\usepackage{amsmath}
\usepackage{amssymb}
\usepackage{mathtools}
\usepackage{stix2}
\usepackage{microtype}

\usepackage{graphicx}
\usepackage{booktabs}
\usepackage{array}
\usepackage{longtable}

\usepackage[numbers,sort&compress]{natbib}
\usepackage[
colorlinks=true,
linkcolor=black,
citecolor=black,
urlcolor=black,
pdfauthor={Aryaman Chandra, Sudhir Ranjan Jain},
pdftitle={On the Spectral Properties of Discrete Landscape Functions},
pdfsubject={mathematics}
]{hyperref}
\theoremstyle{plain}
\newtheorem{theorem}{Theorem}[section]

\newtheorem{proposition}[theorem]{Proposition}
\newtheorem{corollary}[theorem]{Corollary}
\newtheorem{conjecture}[theorem]{Conjecture}

\theoremstyle{definition}

\theoremstyle{remark}

\newtheorem*{remark*}{Remark}

\newcommand{\Q}{\mathbb{Q}}

\title{ON THE SPECTRAL PROPERTIES OF DISCRETE LANDSCAPE FUNCTIONS}

\author{Aryaman Chandra}
\author{Sudhir Ranjan Jain}



\date{}

\begin{document}
	
	\begin{abstract}
We study the spectral coefficients of the discrete localization landscape on an interval. The coefficients have an exact trigonometric form and lie in explicit abelian extensions of $\mathbb Q$. Galois symmetry gives the degree bound
\[
[\mathbb Q(c_k(N)):\mathbb Q]\leq\varphi(N)
\]
for every odd mode $k$. The \emph{Chandra--Jain degree conjecture}, which asserts that $[\mathbb Q(c_1(N)):\mathbb Q]=\varphi(N)$ for all $N\geq3$, was subsequently proved by Q.~Zhou~\citep{Zhou2026}. Using this theorem, we show that every odd mode coprime to $N$ has degree $\varphi(N)$. For non-coprime modes, the reduction
\[
c_k(N)=\frac{1}{\sqrt d}\,c_{k/d}(N/d),\qquad d=\gcd(k,N),
\]
gives an upper bound of $2\varphi(N/d)$. Exact computations for all $624$ pairs with $3\leq N\leq50$ and odd $k<N$ yield only the degrees $\varphi(N/d)$ and $2\varphi(N/d)$ in the non-coprime cases. This motivates the Second Chandra--Jain Conjecture, which asserts that this degree dichotomy holds for every non-coprime odd mode. For odd prime $p$, each normalized coefficient is an algebraic integer whose principal ideal is the unique prime ideal above $p$ in $\mathbb Z[\zeta_p]$.
\end{abstract}
	
	\maketitle
	%==============================================================
	\section{Introduction}
	\label{sec:intro}

The localization landscape introduced by Filoche and Mayboroda is defined by the boundary-value problem $Lu=1$ with Dirichlet boundary conditions; its relation to eigenfunction localization has since been studied for continuous and discrete operators \citep{FilocheMayboroda2011,FilocheMayboroda2012,ArnoldEffective2016,ArnoldComputing2019,HarrellMaltsev2020,FilocheMayborodaTao2021,Mugnolo2024}. We study the landscape for the one-dimensional Dirichlet Laplacian and the arithmetic properties of its discrete spectral coefficients.

For $L=-d^2/dx^2$ on $(0,1)$, the landscape is $u(x)=x(1-x)/2$. In the normalized eigenbasis $\psi_n(x)=\sqrt{2}\sin(n\pi x)$, its coefficients are
\[
 b_n=\frac{\sqrt{2}}{(n\pi)^3}\bigl(1-(-1)^n\bigr),
\]
so $b_n=0$ for even $n$. On the grid $j=0,\ldots,N$, the unscaled discrete Dirichlet Laplacian has landscape $u_N(j)=j(N-j)/2$ and normalized eigenvectors $\psi_{N,k}(j)=\sqrt{2/N}\sin(k\pi j/N)$. We derive the exact coefficient formula
\[
 c_k(N)=\frac14\sqrt{\frac2N}\cot\left(\frac{k\pi}{2N}\right)\csc^2\left(\frac{k\pi}{2N}\right)
 \quad (k\text{ odd}),
\]
with $c_k(N)=0$ for even $k$.

The trigonometric formula places the coefficients in explicit cyclotomic extensions. A Galois action gives $[\mathbb Q(c_k(N)):\mathbb Q]\leq\varphi(N)$. The Chandra--Jain degree conjecture, asserting $[\mathbb Q(c_1(N)):\mathbb Q]=\varphi(N)$ for all $N\geq3$, was subsequently proved by Q.~Zhou~\citep{Zhou2026}. It follows that every odd mode coprime to $N$ has degree $\varphi(N)$. For non-coprime modes, we prove the reduction $c_k(N)=d^{-1/2}c_{k/d}(N/d)$, where $d=\gcd(k,N)$, and obtain the upper bound $2\varphi(N/d)$. Exact computations for all $624$ pairs with $3\leq N\leq50$ show that only $\varphi(N/d)$ and $2\varphi(N/d)$ occur in the non-coprime cases, motivating the Second Chandra--Jain Conjecture formulated below. We also establish an algebraic-integer property for normalized coefficients at odd prime levels.

	\section{Continuous and Discrete Landscapes}
	\label{sec:landscape}

	\subsection{Continuous model}
	\label{sec:continuous}

Let $L=-d^2/dx^2$ on $(0,1)$ with Dirichlet boundary conditions. The landscape solves
\[
 -u''(x)=1,\qquad u(0)=u(1)=0,
\]
and hence
\[
 u(x)=\frac{x(1-x)}2.
\]
The normalized Dirichlet eigenfunctions are $\psi_n(x)=\sqrt2\sin(n\pi x)$, with eigenvalues $(n\pi)^2$. Writing $u=\sum_{n\geq1}b_n\psi_n$, two integrations by parts yield
\[
 b_n=\int_0^1u(x)\psi_n(x)\,dx
 =\frac{\sqrt2}{(n\pi)^3}\bigl(1-(-1)^n\bigr).
\]
Thus $b_n=0$ for even $n$ and $b_n=2\sqrt2/(n\pi)^3$ for odd $n$. This vanishing also follows from $u(1-x)=u(x)$ and $\psi_n(1-x)=(-1)^{n+1}\psi_n(x)$.

	\subsection{Finite setup}
	\label{sec:finite-setup}

For $N\geq2$, let $x_j=j/N$ for $j=0,\ldots,N$, with $u_0=u_N=0$. Use the unscaled discrete Dirichlet Laplacian
\[
 (L_Nv)_j=2v_j-v_{j-1}-v_{j+1},\qquad 1\leq j\leq N-1.
\]
The discrete landscape solves $L_Nu_N=\mathbf 1$. Direct substitution gives
\[
 u_N(j)=\frac{j(N-j)}2,\qquad 0\leq j\leq N,
\]
so $u_N(j)/N^2=u(j/N)$ and $u_N(N-j)=u_N(j)$.

	\subsection{Discrete modes and reflection}
	\label{sec:parity}

The normalized eigenvectors of $L_N$ are
\[
 \psi_{N,k}(j)=\sqrt{\frac2N}\sin\left(\frac{k\pi j}{N}\right),
 \qquad 1\leq k\leq N-1,
\]
with eigenvalues $\mu_k=4\sin^2(k\pi/(2N))$. They form an orthonormal basis for the interior grid values under
\[
 \langle f,g\rangle=\sum_{j=1}^{N-1}f(j)g(j).
\]
Define the spectral coefficients by
\[
 u_N(j)=\sum_{k=1}^{N-1}c_k(N)\psi_{N,k}(j),\qquad
 c_k(N)=\langle u_N,\psi_{N,k}\rangle.
\]
The reflection identity
\[
 \psi_{N,k}(N-j)=(-1)^{k+1}\psi_{N,k}(j)
\]
and the symmetry $u_N(N-j)=u_N(j)$ imply $c_k(N)=(-1)^{k+1}c_k(N)$. Consequently, $c_k(N)=0$ for every even $k$.

	\section{Exact Spectral Coefficients}
	\label{sec:coefficients}
	%==============================================================
	
	The finite setup reduces the problem to an explicit calculation. The operator $L_N$ is symmetric with respect to the dot product. Using $L_Nu_N=\mathbf 1$ and $L_N\psi_{N,k}=\mu_k\psi_{N,k}$, we obtain
	\[
	\mu_k c_k(N)
	=\langle u_N,L_N\psi_{N,k}\rangle
	=\langle L_Nu_N,\psi_{N,k}\rangle
	=\langle \mathbf 1,\psi_{N,k}\rangle,
	\]
	where $\mathbf 1$ denotes the vector whose interior entries are all equal to $1$. Thus
	\[
	\mu_k c_k(N)
	=
	\sqrt{\frac2N}
	\sum_{j=1}^{N-1}
	\sin\left(\frac{k\pi j}{N}\right).
	\]
	
	The remaining sum is elementary. For odd $k$,
	\[
	\sum_{j=1}^{N-1}
	\sin\left(\frac{k\pi j}{N}\right)
	=
	\cot\left(\frac{k\pi}{2N}\right),
	\]
	while for even $k$ it vanishes, consistently with the symmetry argument above. Since
	\[
	\mu_k=4\sin^2\left(\frac{k\pi}{2N}\right),
	\]
	we obtain the following exact formula.
	
	\begin{theorem}[Exact spectral coefficients]
		\label{thm:main}
		For $1\le k<N$,
		\[
		c_k(N)=0
		\qquad\text{if $k$ is even},
		\]
		and, for odd $k$,
		\[
		c_k(N)
		=
		\frac14\sqrt{\frac2N}
		\cot\left(\frac{k\pi}{2N}\right)
		\csc^2\left(\frac{k\pi}{2N}\right).
		\]
	\end{theorem}
	
	This formula is the main bridge from the differential problem to arithmetic. A system of $N-1$ coupled equations has produced numbers involving only elementary trigonometric functions at rational multiples of $\pi$.
	
	\begin{proposition}[Monotonicity of odd coefficients]
		\label{prop:odd-monotonicity}
		For fixed $N\ge3$, the nonzero spectral coefficients are positive and strictly decreasing with the odd mode,
		\[
		c_1(N)>c_3(N)>c_5(N)>\cdots>0.
		\]
	\end{proposition}
	
	\begin{proof}
		For odd $k<N$, put
		\[
		x_k=\frac{k\pi}{2N}\in\left(0,\frac\pi2\right),
		\qquad
		f(x)=\cot x\,\csc^2x.
		\]
		The theorem gives
		\[
		c_k(N)=\frac14\sqrt{\frac2N}\,f(x_k).
		\]
		Since
		\[
		f'(x)=-\frac{\sin^2x+3\cos^2x}{\sin^4x}<0,
		\]
		and $x_k$ increases with $k$, the stated inequalities follow.
	\end{proof}
	
	\begin{figure}[htbp]
		\centering
		\includegraphics[width=0.94\textwidth]{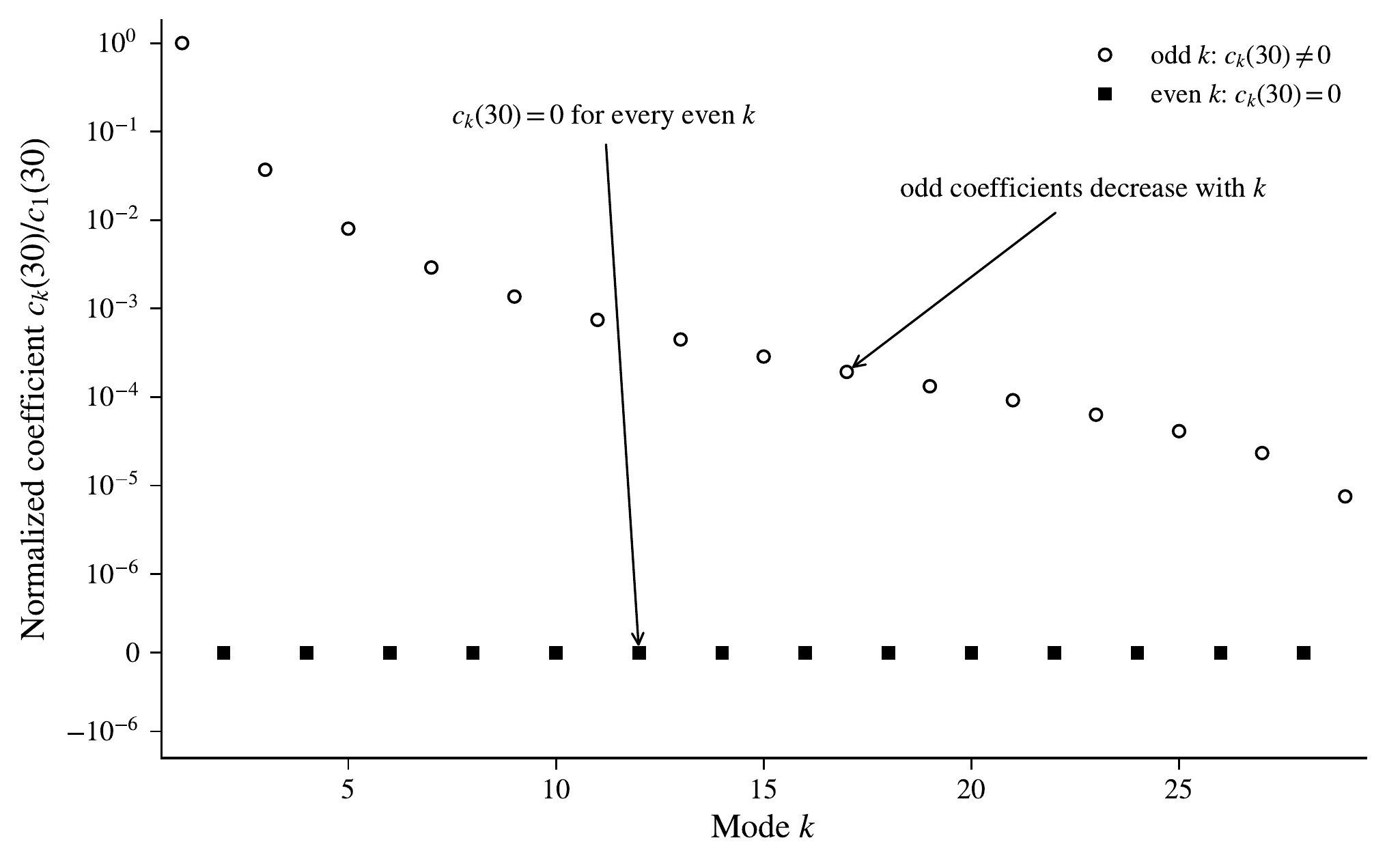}
		\caption{Exact spectral coefficient profile for $N=30$. The odd coefficients are positive and strictly decreasing with the mode, while every even coefficient is exactly zero by reflection symmetry. The vertical scale is normalized by $c_1(30)$ and displayed on a symmetric-logarithmic scale so that the full odd-mode range remains visible.}
		\label{fig:spectral-profile}
	\end{figure}
	
	%==============================================================
	\section{Cyclotomic Structure and Degree Bound}
	\label{sec:cyclotomic}
	%==============================================================
	
	The coefficient formula has an elementary trigonometric appearance, but its angles are rational multiples of $\pi$. This brings roots of unity into the problem and places the coefficients inside cyclotomic fields~\citep{Washington}. We now make that connection explicit.
	
	\subsection{Reaching roots of unity}
	\label{subsec:ambient-field}
	
	For odd $k$, write
	\[
	A_{N,k}
	=
	\cot\left(\frac{k\pi}{2N}\right)
	\csc^2\left(\frac{k\pi}{2N}\right),
	\qquad
	B_N=\sqrt{\frac2N}.
	\]
	Then
	\[
	c_k(N)=\frac14A_{N,k}B_N.
	\]
	Let
	\[
	\zeta=\zeta_{4N}=e^{2\pi i/(4N)}.
	\]
	Since
	\[
	e^{ik\pi/N}=\zeta^{2k},
	\]
	the identities
	\[
	\cot\left(\frac{k\pi}{2N}\right)
	=
	i\frac{\zeta^{2k}+1}{\zeta^{2k}-1}
	\]
	and
	\[
	\csc^2\left(\frac{k\pi}{2N}\right)
	=-\frac{4\zeta^{2k}}{(\zeta^{2k}-1)^2}
	\]
	show that $A_{N,k}$ is a rational function of a root of unity. Since it is real,
	\[
	A_{N,k}\in\mathbb Q(\zeta_{4N})^+.
	\]
	The remaining factor satisfies
	\[
	B_N^2=\frac2N\in\mathbb Q.
	\]
	Put $F_N=\mathbb Q(\zeta_{4N})^+$. Since
	\[
	[F_N:\mathbb Q]=\frac{\varphi(4N)}2
	=\begin{cases}
	\varphi(N),&N\text{ odd},\\
	2\varphi(N),&N\text{ even},
	\end{cases}
	\]
	we can determine whether adjoining $B_N$ enlarges the field by examining its quadratic field:
	\[
	\mathbb Q(B_N)=\mathbb Q(\sqrt{2N}).
	\]
	If $N$ is even, the conductor of this quadratic field divides $2N$, so it is contained in $F_N$. If $N$ is odd, its conductor has $2$-adic valuation $3$, whereas $4N$ has $2$-adic valuation $2$, so it is not contained in $F_N$ and adjoining $B_N$ doubles the degree. Thus
	\[
	K_N:=F_N(B_N),\qquad [K_N:\mathbb Q]=2\varphi(N)
	\]
	for every $N\ge2$. Since $F_N/\mathbb Q$ and $\mathbb Q(B_N)/\mathbb Q$ are abelian, their compositum $K_N/\mathbb Q$ is abelian.
	
	Equivalently, if
	\[
	\omega_{N,k}=e^{ik\pi/N},
	\]
	then
	\[
	c_k(N)
	=
	i\sqrt{\frac2N}\,
	\frac{\omega_{N,k}(\omega_{N,k}+1)}{(1-\omega_{N,k})^3}.
	\]
	Thus each coefficient is a rational expression in a root of unity together with the elementary factor $\sqrt{2/N}$.
	
	\begin{proposition}[Cyclotomic form]
		\label{prop:cyclotomic-form}
		For odd $k$, the coefficient $c_k(N)$ belongs to $K_N$ and is given by
		\[
		c_k(N)
		=
		i\sqrt{\frac2N}\,
		\frac{\omega_{N,k}(\omega_{N,k}+1)}{(1-\omega_{N,k})^3}.
		\]
	\end{proposition}
	
	\begin{proof}
		Writing $x=k\pi/(2N)$ gives $e^{2ix}=\omega_{N,k}$. The identities
		\[
		\cot x=i\frac{e^{2ix}+1}{e^{2ix}-1},
		\qquad
		\csc^2x=-\frac{4e^{2ix}}{(e^{2ix}-1)^2}
		\]
		substituted into Theorem~\ref{thm:main} give the stated formula.
	\end{proof}
	
	\subsection{Galois action}
	\label{subsec:galois}
	
	The roots of unity come with natural symmetries. For
	\[
	a\in(\mathbb Z/4N\mathbb Z)^\times,
	\]
	let $\sigma_a$ be the automorphism
	\[
	\sigma_a(\zeta_{4N})=\zeta_{4N}^a.
	\]
	These automorphisms permute the roots of unity and therefore permute their trigonometric expressions. Writing
	\[
	\delta(a)=(-1)^{(a-1)/2},
	\]
	direct substitution gives
	\[
	\sigma_a(A_{N,k})
	=
	\delta(a)A_{N,ak},
	\]
	where the index $ak$ is interpreted using the periodicity and sign change of $A_{N,k}$. Since $B_N^2\in\mathbb Q$,
	\[
	\sigma_a(B_N)=\pm B_N.
	\]
	Thus every algebraic conjugate of $c_k(N)$ is obtained from the same family of coefficients, up to sign.
	
	\subsection{Fixed involution and degree bound}
	\label{subsec:degree-bound}
	
	Consider the residue class
	\[
	a=2N-1.
	\]
	It satisfies
	\[
	(2N-1)^2\equiv1\pmod{4N},
	\]
	so its restriction to $\mathbb Q(\zeta_{4N})^+$ has order two. For odd $k$,
	\[
	\sigma_a(A_{N,k})=\delta(a)A_{N,ak}.
	\]
	Since $ak\equiv2N-k\pmod{4N}$ and $A_{N,2N-k}=-A_{N,k}$, we obtain
	\[
	\sigma_a(A_{N,k})=(-1)^NA_{N,k}.
	\]
	When $N$ is even, $K_N=F_N$ and $B_N$ lies in this real cyclotomic field. Its quadratic field has conductor dividing $2N$, and $a\equiv-1\pmod{2N}$, so $\sigma_a$ restricts to complex conjugation on $\mathbb Q(B_N)$ and fixes $B_N$; set $\tau_N=\sigma_a$. When $N$ is odd, $K_N/F_N$ is quadratic; choose the lift $\tau_N$ of $\sigma_a|_{F_N}$ satisfying $\tau_N(B_N)=-B_N$. Thus, in both cases,
	\[
	\tau_N(c_k(N))=c_k(N)
	\]
	for every odd $k$.
	
	The automorphism $\tau_N$ is nontrivial because
	\[
	\tau_N(\zeta_{4N})
	=\zeta_{4N}^{2N-1}
	=-\zeta_{4N}^{-1}
	\neq\zeta_{4N}.
	\]
	Its fixed field
	\[
	E_N=\{x\in K_N:\tau_N(x)=x\}
	\]
	has index two in $K_N$. Therefore
	\[
	[E_N:\mathbb Q]
	=\frac{[K_N:\mathbb Q]}2
	=\varphi(N).
	\]
	Every odd coefficient lies in $E_N$, so
	\[
	\mathbb Q(c_k(N))\subseteq E_N
	\]
	and hence
	\[
	[\mathbb Q(c_k(N)):\mathbb Q]\leq\varphi(N).
	\]
	
	\begin{corollary}[Degree bound]
		\label{cor:degree-upper-bound}
		For every $N\ge2$ and every odd $k<N$,
		\[
		[\mathbb Q(c_k(N)):\mathbb Q]\leq\varphi(N).
		\]
	\end{corollary}
	
	\subsection{Prime level.}
	\label{subsec:prime-level}
	
	The cyclotomic description becomes particularly concrete when the discretization size is an odd prime $p$. For odd $k<p$, write
	\[
	r\equiv k\frac{p+1}{2}\pmod p,
	\qquad
	1\le r<p.
	\]
	Since $\zeta_{2p}^k=-\zeta_p^r$,
	\[
	c_k(p)
	=-i\sqrt{\frac2p}
	\frac{\zeta_p^r(1-\zeta_p^r)}{(1+\zeta_p^r)^3}.
	\]
	The denominator is a cyclotomic unit, while $1-\zeta_p^r$ is associated with the unique prime ideal above $p$~\citep{Washington}. Define
	\[
	\alpha_{p,k}:=\frac{c_k(p)}{i\sqrt{2/p}}.
	\]
	Then
	\[
	\alpha_{p,k}
	=-\zeta_p^r\frac{1-\zeta_p^r}{(1+\zeta_p^r)^3}
	\in\mathbb Z[\zeta_p],
	\]
	and
	\[
	(\alpha_{p,k})=(1-\zeta_p),
	\qquad
	N_{\mathbb Q(\zeta_p)/\mathbb Q}(\alpha_{p,k})=p.
	\]
	Thus every odd mode at a prime discretization size yields, after removing the elementary factor $i\sqrt{2/p}$, an algebraic integer associated with the unique prime ideal above $p$.
	
	%==============================================================
	\section{The First Coefficient}
	\label{sec:first}
	%==============================================================
	
	\subsection{Maximal degree.}
	\label{sec:monotonicity}
	
	The degree bound gives an upper limit, but it does not tell us whether an odd coefficient actually reaches that bound. The first coefficient is the natural place to begin. Write
	\[
	f(x)=\cot x\,\csc^2x=\frac{\cos x}{\sin^3x}.
	\]
	Then
	\[
	c_1(N)
	=
	\frac14\sqrt{\frac2N}
	f\left(\frac{\pi}{2N}\right).
	\]
	On $(0,\pi)$,
	\[
	f'(x)
	=-\frac{\sin^2x+3\cos^2x}{\sin^4x}<0,
	\]
	and
	\[
	f(\pi-x)=-f(x).
	\]
	The strict monotonicity helps explain why the conjecture is plausible, but it does not by itself prove the degree statement. The theorem quoted below is due to Zhou~\citep{Zhou2026}.
	
	\subsection{Exact computational evidence}
	\label{subsec:exact-verification}
	
	To compute the minimal polynomial of $c_1(N)$, set
	\[
	t=\cot\left(\frac{\pi}{2N}\right).
	\]
	Since $\csc^2x=1+\cot^2x$,
	\[
	c_1(N)
	=
	\frac14\sqrt{\frac2N}\,t(1+t^2),
	\]
	and hence
	\[
	8Nc_1(N)^2=(t+t^3)^2.
	\]
	Also,
	\[
	Q_{2N}(t)=\operatorname{Im}(t+i)^{2N}
	\]
	vanishes at $t=\cot(\pi/(2N))$. We factor $Q_{2N}$ over $\mathbb Q$, eliminate $t$ from
	\[
	Q_{2N}(t)=0,
	\qquad
	8Ny-(t+t^3)^2=0,
	\]
	and then replace $y$ by $x^2$. Factoring the resulting polynomial over $\mathbb Q$ gives the minimal polynomial of $c_1(N)$.
	
	The calculation was carried out for
	\[
	3\leq N\leq30.
	\]
	All $28$ cases have degree $\varphi(N)$.
	
	\begin{table}[htbp]
		\centering
		\renewcommand{\arraystretch}{1.15}
		\begin{tabular}{c|rrrrrrrrrrrrrr}
			$N$ & 3&4&5&6&7&8&9&10&11&12&13&14&15&16\\
			\hline
			$\deg_{\Q}c_1(N)$&2&2&4&2&6&4&6&4&10&4&12&6&8&8\\
			$\varphi(N)$&2&2&4&2&6&4&6&4&10&4&12&6&8&8
		\end{tabular}
		
		\vspace{3pt}
		
		\begin{tabular}{c|rrrrrrrrrrrrrr}
			$N$ & 17&18&19&20&21&22&23&24&25&26&27&28&29&30\\
			\hline
			$\deg_{\Q}c_1(N)$&16&6&18&8&12&10&22&8&20&12&18&12&28&8\\
			$\varphi(N)$&16&6&18&8&12&10&22&8&20&12&18&12&28&8
		\end{tabular}
		\caption{Exact computational verification of $\deg_{\mathbb Q}c_1(N)=\varphi(N)$ for $3\leq N\leq30$.}
		\label{tab:degree-verification}
	\end{table}
	
	For example, the exact minimal polynomial at $N=5$ is
	\[
	5x^4-130x^2+4,
	\]
	which is irreducible over $\mathbb Q$. At $N=13$, the minimal polynomial has degree $12$. As a separate check, the calculation was repeated for
	\[
	N=3,4,5,7,8,9
	\]
	using a direct primitive-root-of-unity representation and an independent resultant computation. The two exact constructions agree in all six cases. The complete computational procedure is described in Appendix~\ref{app:computation}.
	
	\subsection{The Chandra--Jain degree conjecture and Zhou's theorem}
	\label{subsec:conjecture}
	
	The computations originally led us to the following statement, which we named the Chandra--Jain degree conjecture:
	\begin{conjecture}[Chandra--Jain degree conjecture]
		\label{conj:chandra-jain}
		For every integer $N\ge3$,
		\[
		[\mathbb Q(c_1(N)):\mathbb Q]=\varphi(N).
		\]
	\end{conjecture}

	Q.~Zhou subsequently proved the conjecture. We record the result in the form used below.
	\begin{theorem}[Zhou]
		\label{thm:zhou-degree}
		For every integer $N\ge3$,
		\[
		[\mathbb Q(c_1(N)):\mathbb Q]=\varphi(N).
		\]
	\end{theorem}

	Zhou's proof identifies the stabilizer of $c_1(N)$ in $\operatorname{Gal}(K_N/\mathbb Q)$ as a subgroup of order two. Since
	\[
	[K_N:\mathbb Q]=2\varphi(N),
	\]
	the orbit--stabilizer theorem gives the stated degree~\citep{Zhou2026}. We quote Zhou's result rather than reproduce that proof here; the computations in the preceding subsection provide independent exact checks for $3\leq N\leq30$.

	%==============================================================
	\section{Higher Modes}
	\label{sec:higher}
	%==============================================================
	
	The first coefficient is only the beginning. For a general odd mode, the answer depends on the arithmetic relation between the mode number $k$ and the discretization size $N$.
	
	\subsection{Coprime modes}
	\label{sec:coprime}
	
	Suppose $k$ is odd and coprime to $N$. Then
	\[
	\gcd(k,4N)=1,
	\]
	so
	\[
	\zeta_{4N}\longmapsto\zeta_{4N}^{\,k}
	\]
	defines a Galois automorphism $\sigma_k$ of $K_N$. The Galois action gives
	\[
	\sigma_k(A_{N,1})
	=\delta(k)A_{N,k},
	\qquad
	\delta(k)=(-1)^{(k-1)/2},
	\]
	and
	\[
	\sigma_k(B_N)=\pm B_N.
	\]
	Therefore
	\[
	\sigma_k(c_1(N))=\pm c_k(N).
	\]
	Since $\sigma_k$ fixes $\mathbb Q$,
	\[
	[\mathbb Q(c_k(N)):\mathbb Q]
	=
	[\mathbb Q(c_1(N)):\mathbb Q]
	=\varphi(N)
	\]
	by Theorem~\ref{thm:zhou-degree}.
	
	\begin{corollary}[Coprime higher-mode degree]
		\label{cor:coprime-higher-degree}
		Let $N\ge3$, and let $k$ be odd with
		\[
		1\le k<N,
		\qquad
		\gcd(k,N)=1.
		\]
		Then
		\[
		[\mathbb Q(c_k(N)):\mathbb Q]=\varphi(N).
		\]
	\end{corollary}
	
	There is a stronger field-theoretic statement. If
	\[
	E_N=\{x\in K_N:\tau_N(x)=x\}
	\]
	is the fixed field used in the degree bound, then
	\[
	\mathbb Q(c_k(N))=E_N
	\]
	for every odd $k$ coprime to $N$. Thus all such modes generate the same degree-$\varphi(N)$ subfield of $K_N$.
	
	In particular, if $p\ge3$ is prime, every odd mode $1\le k<p$ is coprime to $p$, and therefore
	\[
	[\mathbb Q(c_k(p)):\mathbb Q]=p-1.
	\]
	
	\subsection{Non-coprime modes}
	\label{sec:noncoprime}
	
	The situation changes when $k$ and $N$ share a common factor. Put
	\[
	d=\gcd(k,N),
	\qquad
	N=dN_0,
	\qquad
	k=dk_0,
	\]
	so that
	\[
	\gcd(k_0,N_0)=1.
	\]
	Substitution into the exact coefficient formula gives the simple reduction
	\[
	c_k(N)=\frac1{\sqrt d}\,c_{k_0}(N_0).
	\]
Put $\alpha=c_{k_0}(N_0)$ and $\beta=c_k(N)$. The coprime result gives
	\[
	[\mathbb Q(\alpha):\mathbb Q]=\varphi(N_0).
	\]
	This also holds when $N_0=2$, since then $k_0=1$ and $c_1(2)=1/2$. Writing $m=\varphi(N_0)$, the reduction gives $\beta=\alpha/\sqrt d$, hence
	\[
	\beta^2=\frac{\alpha^2}{d},
		\qquad
		\mathbb Q(\beta^2)=\mathbb Q(\alpha^2).
	\]
	Because both $\alpha$ and $\beta$ satisfy quadratic equations over the fields generated by their squares, it follows that
	\[
	[\mathbb Q(\beta):\mathbb Q]
	\in\left\{\frac m2,m,2m\right\},
		\qquad m=\varphi(N_0),
	\]
	where $m/2$ can occur only when it is an integer. In particular,
	\[
	[\mathbb Q(\beta):\mathbb Q]\leq2\varphi(N_0).
	\]
	The computations below show that the potential half-degree $m/2$ never occurs for the tested non-coprime modes.
	
	\subsection{Computations}
	\label{sec:survey}
	
	We computed the minimal polynomial of every odd coefficient $c_k(N)$ for
	\[
	3\leq N\leq50,
	\qquad
	1\leq k<N,
	\qquad
	k\text{ odd}.
	\]
	There are $624$ such pairs $(N,k)$. Among the $517$ coprime cases,
	\[
	[\mathbb Q(c_k(N)):\mathbb Q]=\varphi(N)
	\]
	in every case. Among the $107$ non-coprime cases, $59$ have degree
	\[
	\varphi(N_0),
	\]
	and $48$ have degree
	\[
	2\varphi(N_0).
	\]
	No other degree occurs in the computed range.
	
	\begin{table}[htbp]
		\centering
		\renewcommand{\arraystretch}{1.15}
		\begin{tabular}{lrr}
			\toprule
			Case & Number of cases & Observed degree \\
			\midrule
			All odd modes & 624 & -- \\
			$\gcd(N,k)=1$ & 517 & $\varphi(N)$ \\
			$\gcd(N,k)>1$ & 59 & $\varphi(N_0)$ \\
			$\gcd(N,k)>1$ & 48 & $2\varphi(N_0)$ \\
			\bottomrule
		\end{tabular}
		
		\caption{Exact higher-mode degree survey for all odd $k<N$ and $3\le N\le50$.}
		\label{tab:higher-mode-summary}
	\end{table}
	
	\begin{figure}[htbp]
		\centering
		\includegraphics[width=0.94\textwidth]{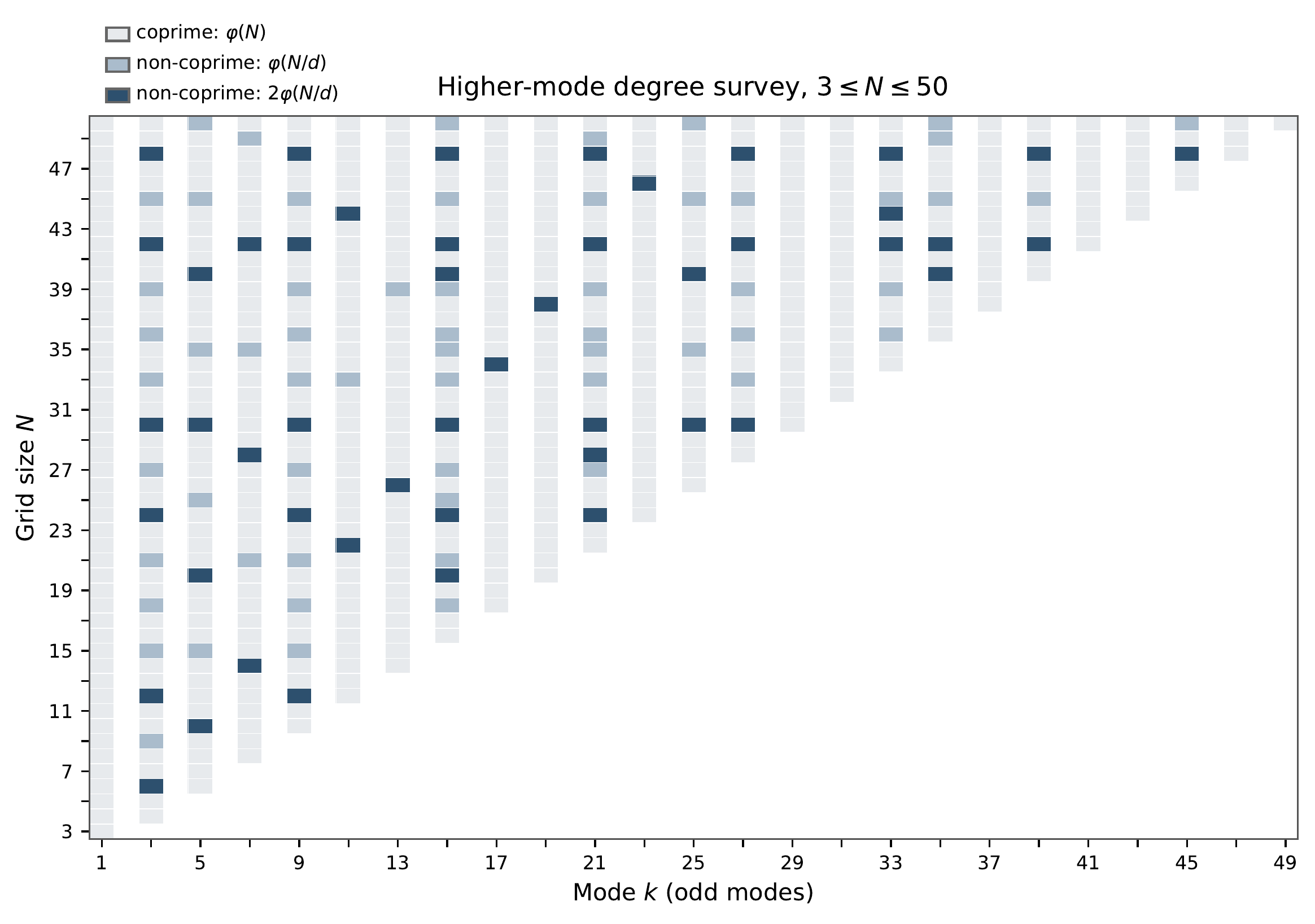}
		\caption{Exact higher-mode degree survey for the $624$ computed cases. The occupied levels distinguish coprime cases, non-coprime cases of degree $\varphi(N/d)$, and non-coprime cases of degree $2\varphi(N/d)$.}
		\label{fig:higher-mode-heatmap}
	\end{figure}
	
	Representative non-coprime cases are shown below.
	
	\begin{table}[htbp]
		\centering
		\renewcommand{\arraystretch}{1.15}
		\begin{tabular}{c c c c c c}
			\toprule
			$N$ & $k$ & $d$ & $N_0$ & $[\mathbb Q(c_k(N)):\mathbb Q]$ & $\deg/\varphi(N_0)$ \\
			\midrule
			6  & 3  & 3  & 2  & 2 & 2 \\
			9  & 3  & 3  & 3  & 2 & 1 \\
			12 & 3  & 3  & 4  & 4 & 2 \\
			15 & 3  & 3  & 5  & 4 & 1 \\
			20 & 5  & 5  & 4  & 4 & 2 \\
			21 & 7  & 7  & 3  & 2 & 1 \\
			34 & 17 & 17 & 2  & 2 & 2 \\
			50 & 25 & 25 & 2  & 1 & 1 \\
			\bottomrule
		\end{tabular}
		\caption{Representative exact non-coprime higher-mode degrees.}
		\label{tab:higher-mode-examples}
	\end{table}
	
	For example,
	\[
	c_3(9)=\frac{\sqrt6}{3},
	\qquad
	[\mathbb Q(c_3(9)):\mathbb Q]=2=\varphi(3),
	\]
	whereas
	\[
	[\mathbb Q(c_3(30)):\mathbb Q]=8=2\varphi(10).
	\]
	
	The computation uses exact algebraic arithmetic. The complete dataset contains the minimal polynomial of each of the $624$ coefficients.
	
	\subsection{Open problems}
	\label{sec:open}
	
The computation gives two degree values in the non-coprime cases through $N=50$. This motivates the following conjecture.

	\begin{conjecture}[Chandra--Jain Dichotomy]
	\label{conj:chandra-jain-second}
	Let $N\geq3$ and let $k$ be odd with $1\leq k<N$ and $d=\gcd(N,k)>1$. Set $N_0=N/d$. Then
	\[
	[\mathbb Q(c_k(N)):\mathbb Q]
	\in\left\{\varphi(N_0),\,2\varphi(N_0)\right\}.
	\]
	\end{conjecture}

	The conjecture is consistent with all $107$ non-coprime cases computed for $3\leq N\leq50$: $59$ have degree $\varphi(N_0)$ and $48$ have degree $2\varphi(N_0)$. No proof for general $N$ is known here. Equivalently, in view of the quadratic reduction above, the conjecture asserts that the possible degree $\varphi(N_0)/2$ never occurs. Determining an arithmetic criterion for which of the two conjectured degrees occurs remains open. The dichotomy arises from the common-factor reduction $d=\gcd(N,k)$ and the possible contribution of the additional quadratic factor $\sqrt d$.
	
	Two further extensions are open. One is to replace the interval by other one-dimensional domains. The other is to determine which parts of the reflection and arithmetic structure remain for higher-dimensional landscape functions.
	
	%==============================================================
	\appendix
	\section{Exact Computation}
	\label{app:computation}
	%==============================================================
	
	The first-coefficient computation uses
	\[
	t=\cot\left(\frac{\pi}{2N}\right),
	\qquad
	Q_{2N}(t)=\operatorname{Im}(t+i)^{2N}.
	\]
	The polynomial $Q_{2N}$ has integer coefficients and vanishes at $t$. We factor it over $\mathbb Q$, select the factor containing the required root, and eliminate $t$ from
	\[
	Q_{2N}(t)=0,
	\qquad
	8Ny-(t+t^3)^2=0.
	\]
	After substituting $y=x^2$, exact factorization identifies the minimal polynomial of $c_1(N)$.
	
	For the higher-mode survey, each $c_k(N)$ is represented by the exact trigonometric formula in Theorem~\ref{thm:main}, and its minimal polynomial over $\mathbb Q$ is computed using exact algebraic arithmetic. The reported degree is the degree of that minimal polynomial. The full $624$-case dataset and the corresponding polynomials are supplied separately.
	
	For the first-coefficient computation, high-precision numerical evaluation is used only to identify the relevant factor among finitely many exact factors. No PSLQ, interpolation, or numerical polynomial fitting is used to obtain the reported polynomials.
	
	\section{Additional Degree Tables}
	\label{app:tables}
	
	The following tables record, for each $N$, the number of odd modes, the coprime count, the non-coprime count, and the two observed non-coprime degree levels.
	
	\small
	\begin{longtable}{c r r r r r}
		\caption{Per-$N$ summary of the $624$-case higher-mode survey, $3\leq N\leq26$.}
		\label{tab:perN-a}\\
		\toprule
		$N$ & total & coprime & non-coprime & $\deg=\varphi(N_0)$ & $\deg=2\varphi(N_0)$\\
		\midrule
		\endfirsthead
		\toprule
		$N$ & total & coprime & non-coprime & $\deg=\varphi(N_0)$ & $\deg=2\varphi(N_0)$\\
		\midrule
		\endhead
		3 & 1 & 1 & 0 & 0 & 0 \\
		4 & 2 & 2 & 0 & 0 & 0 \\
		5 & 2 & 2 & 0 & 0 & 0 \\
		6 & 3 & 2 & 1 & 0 & 1 \\
		7 & 3 & 3 & 0 & 0 & 0 \\
		8 & 4 & 4 & 0 & 0 & 0 \\
		9 & 4 & 3 & 1 & 1 & 0 \\
		10 & 5 & 4 & 1 & 0 & 1 \\
		11 & 5 & 5 & 0 & 0 & 0 \\
		12 & 6 & 4 & 2 & 0 & 2 \\
		13 & 6 & 6 & 0 & 0 & 0 \\
		14 & 7 & 6 & 1 & 0 & 1 \\
		15 & 7 & 4 & 3 & 3 & 0 \\
		16 & 8 & 8 & 0 & 0 & 0 \\
		17 & 8 & 8 & 0 & 0 & 0 \\
		18 & 9 & 6 & 3 & 3 & 0 \\
		19 & 9 & 9 & 0 & 0 & 0 \\
		20 & 10 & 8 & 2 & 0 & 2 \\
		21 & 10 & 6 & 4 & 4 & 0 \\
		22 & 11 & 10 & 1 & 0 & 1 \\
		23 & 11 & 11 & 0 & 0 & 0 \\
		24 & 12 & 8 & 4 & 0 & 4 \\
		25 & 12 & 10 & 2 & 2 & 0 \\
		26 & 13 & 12 & 1 & 0 & 1 \\
		\bottomrule
	\end{longtable}
	
	\begin{longtable}{c r r r r r}
		\caption{Per-$N$ summary of the $624$-case higher-mode survey, $27\leq N\leq50$.}
		\label{tab:perN-b}\\
		\toprule
		$N$ & total & coprime & non-coprime & $\deg=\varphi(N_0)$ & $\deg=2\varphi(N_0)$\\
		\midrule
		\endfirsthead
		\toprule
		$N$ & total & coprime & non-coprime & $\deg=\varphi(N_0)$ & $\deg=2\varphi(N_0)$\\
		\midrule
		\endhead
		27 & 13 & 9 & 4 & 4 & 0 \\
		28 & 14 & 12 & 2 & 0 & 2 \\
		29 & 14 & 14 & 0 & 0 & 0 \\
		30 & 15 & 8 & 7 & 0 & 7 \\
		31 & 15 & 15 & 0 & 0 & 0 \\
		32 & 16 & 16 & 0 & 0 & 0 \\
		33 & 16 & 10 & 6 & 6 & 0 \\
		34 & 17 & 16 & 1 & 0 & 1 \\
		35 & 17 & 12 & 5 & 5 & 0 \\
		36 & 18 & 12 & 6 & 6 & 0 \\
		37 & 18 & 18 & 0 & 0 & 0 \\
		38 & 19 & 18 & 1 & 0 & 1 \\
		39 & 19 & 12 & 7 & 7 & 0 \\
		40 & 20 & 16 & 4 & 0 & 4 \\
		41 & 20 & 20 & 0 & 0 & 0 \\
		42 & 21 & 12 & 9 & 0 & 9 \\
		43 & 21 & 21 & 0 & 0 & 0 \\
		44 & 22 & 20 & 2 & 0 & 2 \\
		45 & 22 & 12 & 10 & 10 & 0 \\
		46 & 23 & 22 & 1 & 0 & 1 \\
		47 & 23 & 23 & 0 & 0 & 0 \\
		48 & 24 & 16 & 8 & 0 & 8 \\
		49 & 24 & 21 & 3 & 3 & 0 \\
		50 & 25 & 20 & 5 & 5 & 0 \\
		\bottomrule
	\end{longtable}

\end{document}